\documentclass[12pt]{article}
\usepackage[utf8]{inputenc}
\usepackage[top=2.54cm, bottom=2.54cm, left=2.80cm, right=2.80cm]{geometry}
\usepackage{amsmath}
\usepackage{amssymb}
\usepackage{amsthm}
\usepackage{fancyhdr}
\usepackage{latexsym}
\usepackage{mathrsfs}
\usepackage{wasysym}
\usepackage{fancyhdr}
\usepackage{float}
\usepackage{graphicx}
\usepackage[numbers,sort&compress]{natbib}
\usepackage{setspace}
\allowdisplaybreaks[4]

\newtheorem{theorem}{\bf Theorem}[section]
\newtheorem{lemma}[theorem]{Lemma}
\newtheorem{pro}[theorem]{Propositon}
\newtheorem{Def}[theorem]{Definition}

\newtheorem{corollary}[theorem]{Corollary}

\begin{document}
\begin{spacing}{1.1}
\title{Nodal domain theorems of signed hypergraphs}
\author{Lei Zhang$^{a,b,c,d}$, Yaoping Hou$^{a}$\thanks {Email addresses: shuxuezhanglei@163.com, yphou@hunnu.edu.cn.}\\
\small $^a$Department of Mathematics, Hunan Normal University, Changsha, China\\
\small $^b$Department of Mathematics and Statistics, Qinghai Normal University,\\ \small Xining, China\\
\small $^c$Academy of Plateau, Science and Sustainability, Xining, China\\
\small $^d$The State Key Laboratory of Tibetan Information Processing and Application, \\
\small Xining, China\\}
\date{}
\maketitle
\begin{abstract}
An signed hypergraph is a hypergraph where each vertex-edge incidence is given a label of $+1$ or $-1$. In 2019, Jost and Mulas generalized the normalized combinatorial Laplace operator of graphs to signed hypergraphs. In this paper, we establish nodal domain theorems for the normalized combinatorial Laplace operator in signed hypergraphs. We also obtain a lower bound estimates for the number of strong nodal domains.
\end{abstract}

{\bf AMS}\,: 05C50; 05C65; 05C22


\section{Introduction}

\noindent

 Courant's nodal domain theorem which was proved by Richard Courant \cite{Courant1953} in 1920s, is a basic result in spectral theory with wide applications. The theorem states that the nodal lines of the $k$-th eigenfunction $f_k$ of a self-adjoint second order elliptic differential operator can not divide the domain $D$ into more than $k$ different subdomains. Here the nodal lines are refered to the set of zeros (nodes) of eigenfunctions and the subdomains are now known as the nodal domains. Courant's theorem can be considered as a natural generalization of Sturm's oscillation theorem for second order ODEs that the zeros of the $k$-th eigenfunction of a vibrating string divide the string into exactly $k$ subintervals. There are abundant extensions of Courant's theorem to non-linear operators like $p$-Laplacians,  to Riemannian manifolds and discrete settings, see, e.g., \cite{Chang2017, Chen1976, Jost2021}.

The study of discrete nodal domain theorems on graphs dates back to the work of Gantmacher and Krein \cite{Gantmacher2002}, which contains a discrete analogue of Sturm's theorem for strings. Many of Fiedler's results in 1970s \cite{Fiedler1973, Fiedler1975, Fiedler19751} can be interpreted as discrete nodal domain estimates. The discrete nodal domain theorems for generalized Laplacians were established by Davies, Gladwell, Leydold, and Stadler \cite{Davis2001} in 2001. There are many further advances in this topic, see, e.g., \cite{Bykolu2003, Bykolu2005, Lin2010} and the book \cite{Bykolu2007}.

In the last decade, we have witnessed a revolution in hypergraph theory when different tensors or hypermatrices associated with hypergraphs are studied extensively in \cite{Chung1997, Banerjee2017}. Despite promising progress, some aspects of spectral graph theory cannot be generalised to spectral hypergraph theory using tensors. Most tensor-related problems are NP-hard, as shown in \cite{Hillar2013}. The alternative method for studying a hypergraph is to observe different properties of a hypergraph in terms of the spectra of different connectivity matrices associated with the underlying weighted graph of the hypergraph, see \cite{Rodriguez2003, Rodriguez2009, Bretto2013, Banerjee2021}. In 2019, Jost and Mulas \cite{Jost2019} introduced chemical hypergraphs with the aim of modelling chemical reaction networks. They also introduced two normalized Laplace operators for chemical hypergraphs, the vertex Laplacian $L$ and the hyperedge Laplacian $L^H$, as a generalization of the classical theory introduced by Chung for graphs. We have noticed that $L$ is not necessarily a symmetric matrix, but it is a symmetric operator with respect to the scalar product that we use. Similarly, Reff and Rusnak \cite{Reff2012} introduced signed hypergraphs, that is, a hypergraph where each vertex-edge incidence is given a label of $+1$ or $-1$. Shi \cite{Shi1992} called this type of hypergraph a signed hypergraph and used it to model the constrained via minimization (CVM) problem or two-layer routings.

In 2021, Mulas and Zhang \cite{Mulas2021} shown some new spectral properties of the normalized Laplacian defined for signed hypergraphs, and obtained the signless nodal domain theorem for the normalized Laplacian of the signed hypergraphs. Recently, Ge and Liu \cite{Ge2022} establish nodal domain theorems for arbitrary symmetric matrices. It is natural to ask for the sign version of nodal domain theorems of the signed hypergraphs. In this paper, we solve this problem by introducing proper concepts of strong and weak nodal domains to respect the non-uniform signs of off-diagonal entries of the normalized Laplacian operator for the signed hypergraphs and establishing the corresponding nodal domain theorems.

The contents of this paper are organized as follows. Section 2 provides an overview of the preliminaries needed in order to discuss the main results. The next two sections focus on the normalized Laplace operator. In particular, in Section 3 we give the definitions of the strong and weak nodal domains for the signed hypergraphs and some properties are discussed. In Section 4 we prove the Courant nodal domain theorem for the normalized Laplace operator in the signed hypergraphs. Finally, we show a lower bound estimates for the number of strong nodal domains in Section 5.

\section{Preliminary}

\noindent

A hypergraph $H$ is a pair $(V,E)$, where $E\subseteq 2^{V}$ and $2^{V}$ stands for the set of all subsets of $V$. The elements of $V=V(H)$, labeled as $[n]=\{1,\cdots,n\}$, are referred to as vertices and the elements of $E=E(H)$ are called edges. The corank $cr(H)$ and rank $rk(H)$ of a hypergraph $H=(V, E)$, is defined by $cr(H)=min_{e\in E}|e|$, and $rk(H)=max_{e\in E}|e|$. A hypergraph $H$ is called $k$-uniform hypergraph if $cr(H)=rk(H)=k$. We say two vertices $x, y\in V$ are connected by an edge $e$ if $x, y\in e \in E$. We denote by $deg(x)$ the vertex degree of $x\in V$.

A hypergraph $H$ is called a linear hypergraph(see \cite{Bretto2013}), if each pair of the edges of $H$ has at most one common vertex. In a hypergraph, two vertices $x, y$ are said to be adjacent if there is an edge $e$ that contains both of these vertices and write $x\sim y$. Two edges are said to be adjacent if their intersection is not empty. A vertex $v$ is said to be incident to an edge $e$ if $v\in e$. A path of length $q$ in a hypergraph $H$ is defined to be an alternating sequence of vertices and edges $v_1,e_1,v_2,e_2,\cdots,v_q,e_q,v_{q+1}$ such that

(1) $v_1,\cdots,v_{q+1}$ are all distinct vertices of $H$,

(2) $e_1,\cdots,e_q$ are all distinct edges of $H$,

(3) $v_r,v_{r+1}\in e_r$ for $r=1,\cdots,q$.

If $q\geq 1$ and $v_1=v_{q+1}$, then this path is called a cycle of length $q$. A hypergraph $H$ is connected if there exists a path starting at $v$ and terminating at $u$ for all $v,u\in V$, and is called acyclic if it contains no cycle. The other undefined definitions here can refer to \cite{Berge1973} and \cite{Bretto2013}.

The following definition on induced subhypergraph can be found in \cite{Berge1973}.
\begin{Def}[\cite{Berge1973}]\label{2.3}

Let $H=(V, E)$ be a hypergraph with $E=\{e_1, e_2, \cdots, e_m\}$. For a set $A\subset V$ we call the familily
\begin{equation*}
  E_{A}=\{ e_j\cap A : 1\leq j\leq m, e_j\cap A \neq \emptyset \}
\end{equation*}
the subhypergraph induced by the set $A$.
\end{Def}

We present an overview of the basic definitions regarding signed hypergraphs.

Let $V$ and $E$ be disjoint finite sets whose respective elements are called vertices and edges. An incidence function is a function $\iota : V\times E\rightarrow \mathbb{Z}_{\geq 0}$, while a vertex $v$ and an edge $e$ are said to be incident with respect to $\iota$ if $\iota(v, e)\neq 0$. An incidence is a triple $(v, e, k)$, where $v$ and $e$ are incident and $k \in \{1, 2, 3,\cdots , \iota(v, e)\}$. The value of $\iota(v, e)$ is called the multiplicity of the incidence.

\begin{Def}[\cite{Reff2012}]\label{2.1}
Let $\mathcal{I}$ be the set of incidences determined by $\iota$. An vertex-edge incidence is a function $\sigma : \mathcal{I} \rightarrow \{+1, -1\}$. A signed hypergraph $\Gamma=(H, \sigma)$ is a quadruple $(V, E, \mathcal{I}, \sigma(v, e, k))$, and its underlying hypergraph is $H=(V, E, \mathcal{I})$.
\end{Def}

A signed hypergraph is simple if $\iota(v, e) \leq 1$ for all $v$ and $e$, and for convenience we will write $(v, e)$ instead of $(v, e, 1)$ if $H$ is a simple hypergraph. In this paper, we mainly focus the simple signed hypergraph. Assume that $e=\{v_1, v_2,\cdots , v_l\}$ is an edge of a signed hypergraph $\Gamma=(H, \sigma)$. The sign of the edge $e$ is defined as
\begin{equation*}
  sgn(e)=(-1)^{l-1} \prod_{i=1}^{l} \sigma(v_i, e, k_i).
\end{equation*}

\begin{Def}[\cite{Reff2012}]\label{2.2}
The signed hypergraph $\Gamma=(H, \sigma)$ with the underlying hypergraph $H=(V, E)$ has $t$ connected components if there exist $\Gamma_1=(V_1, H_1), \cdots, \Gamma_t=(V_t, H_t)$ such that:

1. For every $i\in {1, \cdots, t}$, $\Gamma_i$ is a connected hypergraph with $V_i\subseteq V$ and $H_i\subseteq H$;

2. For every $i\in {1, \cdots, t}$, $i\neq j$, $V_i\cap V_j =\emptyset$ and therefore also $H_i\cap H_j =\emptyset$;

3. $\bigcup V_i=V, \bigcup H_i=H$.
\end{Def}

Next we provide an overview of the operators associated to signed hypergraphs.

\begin{Def}[\cite{Reff2012}]\label{2.4}
The adjacency matrix $A=[a_{ij}]$ of a simple signed hypergraph $\Gamma=(H, \sigma)$ is defined by
\begin{equation*}
  a_{ij}=\sum_{e\in E} sgn_e(v_i, v_j),
\end{equation*}
where $sgn_e(v_i, v_j)$ represents the sign of the edge $e$ which contains $v_i$ and $v_j$.
\end{Def}

\begin{Def}[\cite{Reff2012}]\label{2.5}
The $n\times n$ diagonal degree matrix $D : =D(\Gamma)$ is defined by
\[
D_{ij} : =\left\{
      \begin{array}{ll}
      deg(i), \ if \ \ i=j;\\
      0, \ \ \ \ \ \ otherwise.
   \end {array}
   \right.
\]
\end{Def}

A consequence of studying signed hypergraphs is that (hyper)graphs and signed graphs can be viewed as specializations. A graph can be thought of as a signed hypergraph where each edge is contained in two incidences, and exactly one incidence of each edge is signed $+1$. A signed graph can be thought of as a signed hypergraph where each edge is contained in two incidences.

A signed hypergraph can be also viewed as a hypergraph $H=(V, E)$ such that $V$ is a finite set of vertices and $E$ is a set such that every element $h$ in $E$ is a pair of disjoint elements $(h_{+}, h_{-})$ in $2^{V}$, where $h_{+}$ represents the vertex set of vertex-edge incidence $\sigma(v, e, k)=+1$, and $h_{-}$ represents the vertex set of vertex-edge incidence $\sigma(v, e, k)=-1$, respectively. We easily see $h$ as $h_{+}\cup h_{-}$. From this point of view, signed hypergraph is similar with the chemical hypergraphs in \cite{Jost2019} and the oriented hypergraph in \cite{Reff2012}.

\begin{Def}[\cite{Jost2019}]\label{2.6}
Let $C(V)$ be the space of functions $f : V\rightarrow \mathbb{R}$, endowed with the scalar product
\begin{equation*}
  <f, g> : =\sum_{i\in V} deg(i)f(i)g(i).
\end{equation*}
The normalized Laplacian associated to $\Gamma=(H, \sigma)$ is the operator $L : C(V)\rightarrow C(V)$ such that, given $f : V\rightarrow \mathbb{R}$ and given $i\in V$,
\begin{eqnarray*}
  Lf(i) &=& \frac{\sum_{h:i \in h_{+}} (\sum_{i^{\prime} \in h_{+} \ of \ h}f(i^{\prime})-\sum_{j^{\prime} \in h_{-} \ of \ h}f(j^{\prime}))}{deg(i)} \\
  &&  -\frac{\sum_{\hat{h}:i \in h_{-}} (\sum_{\hat{i} \in h_{+} \ of \ \hat{h}}f(\hat{i})-\sum_{\hat{j} \in h_{-} \ of \ \hat{h}}f(\hat{j}))}{deg(i)}.
\end{eqnarray*}
\end{Def}

{\bf Remark 2.13.} Note that, as well as the graph normalized Laplacian, $L$ can be rewritten in a matrix form as
\[L=I-D^{-1}A,\]
where $I$ is the $n\times n$ identity matrix. To see this, observe that, given $f : V\rightarrow \mathbb{R}$ and $i\in V$,
\begin{eqnarray*}
  Lf(i) &=& \frac{\sum_{h:i \in h_{+}} (\sum_{i^{\prime} \in h_{+} \ of \ h}f(i^{\prime})-\sum_{j^{\prime} \in h_{-} \ of \ h}f(j^{\prime}))}{deg(i)} \\
  &&  -\frac{\sum_{\hat{h}:i \in h_{-}} (\sum_{\hat{i} \in h_{+} \ of \ \hat{h}}f(\hat{i})-\sum_{\hat{j} \in h_{-} \ of \ \hat{h}}f(\hat{j}))}{deg(i)}\\
  &=& \frac{deg(i)f(i)- \sum_{i, j \in e\in E} A_{ij}f(j)} {deg(i)}\\
  &=& f(i)-\frac{1}{deg(i)} \sum_{i, j \in e\in E} A_{ij}f(j)
\end{eqnarray*}

Let $L$ be the normalized Laplacian associated to signed hypergraph $\Gamma=(H, \sigma)$. We list its eigenvalues with multiplicity as follows:
\begin{equation*}
  \lambda_1\leq \lambda_2\leq \cdots \leq \lambda_n.
\end{equation*}
Recall the following Courant-Fischer mini-max principle.

\begin{theorem}[\cite{Horn2012}]\label{2.9}
Let $V$ be an $n$-dimensional vector space with a positive definite scalar product $<., .>$. Let $V_k$ be the family of all $k$-dimensional subspaces of $V$. Let $A: V\rightarrow V$ be a self adjoint linear operator. Then the eigenvalues $\lambda_1\geq \cdots \geq \lambda_n$ of $A$ can be obtained by
\begin{equation*}
  \lambda_k=\min\limits_{V_{n-k+1}\in \mathcal{V}_{n-k+1}} \max\limits_{g(\neq 0)V_{n-k+1}} \frac{<Ag, g>}{<g, g>}=\max\limits_{V_{k}\in \mathcal{V}_{k}} \min\limits_{g(\neq 0)V_{k}} \frac{<Ag, g>}{<g, g>}
\end{equation*}
The vectors $g_k$ realizing such a min-max or max-min then are corresponding eigenvectors, and the min-max spaces $\mathcal{V}_{n-k+1}$ are spanned by the eigenvectors for the eigenvalues $\lambda_{n-k+1}, \cdots, \lambda_n$, and analogously, the max-min spaces $\mathcal{V}_{k}$ are spanned by the eigenvectors for the eigenvalues $\lambda_1, \cdots, \lambda_{n-k+1}$. Thus, we also have
\begin{equation}\label{}
  \lambda_k=\min\limits_{g\in V, <g, g_j>=0 \ for \ j=k+1, \cdots, n} \frac{<Ag, g>}{<g, g>}=\max\limits_{g\in V, <g, g_l>=0 \ for \ l=1, \cdots, k-1} \frac{<Ag, g>}{<g, g>}
\end{equation}
In particular,
\begin{equation*}
  \lambda_1= \max\limits_{g\in V} \frac{<Ag, g>}{<g, g>}, \lambda_n=\min\limits_{g\in V} \frac{<Ag, g>}{<g, g>}
\end{equation*}
{\bf Remark.} $\frac{<Ag, g>}{<g, g>}$ is called the Rayleigh quotient of $g$.
\end{theorem}

\section{Strong and weak nodal domains on signed hypergraphs}

\subsection{Definition of strong and weak nodal domains}

\noindent

In this section, we introduce strong and weak nodal domain paths (see Definitions \ref{3.1} and \ref{3.2} below) which further induce two kinds of equivalent relations on the support set of a function $f$. Building upon the corresponding equivalent classes, we define strong and weak nodal domains of a function $f$ (Definition \ref{3.4}).
\begin{Def}(Strong nodal domain path). \label{3.1}
Let $\Gamma=(H, \sigma)$ be a signed hypergraph where $H=(V, E)$ and $f: V\rightarrow \mathbb{R}$ be a function. A sequence $\{x_j\}^{\ell}_{j=1}$ of vertices in  $H$ is called a strong nodal domain path of $f$(an $S$-path for short) if $x_j, x_{j+1}\in e \in E$ and $f(x_j)sgn(e)f(x_{j+1}) > 0$ for each $j=1,2, \cdots, \ell$.
\end{Def}

\begin{Def}(Weak nodal domain path). \label{3.2}
Let $\Gamma=(H, \sigma)$ be a signed hypergraph where $H=(V, E)$ and $f: V\rightarrow \mathbb{R}$ be a function.  A sequence $\{x_j\}^{\ell}_{j=1}$ of vertices in  $H$  is called a weak nodal domain path of $f$(an $W$-path for short), that is, for any two consecutive non-zeros $x_i$ and $x_j$ of $f$, i.e., $f(x_i)\neq 0, f(x_j)\neq 0$, and $f(x_k)=0$ for $i<k<j$ and if $x_i$ and $x_j$ are connected by edges $e_1, e_2, \cdots, e_r$ and $f(x_i)sgn(e_1)sgn(e_2)\cdots sgn(e_r)f(x_j)>0$.



\end{Def}

We remark that every edge containing at most 1 non-zeros of $f$ is a $W$-path.

Using the above two types of paths, we introduce the following two equivalence relations on the support set of a function $f$.

\begin{Def} \label{3.3}
Let $\Gamma=(H, \sigma)$ be a signed hypergraph where $H=(V, E)$ and $f: V\rightarrow \mathbb{R}$ be a function. Let $\Omega=\{v\in V: f(v)\neq 0\}$ be the support set of $f$.

(i) We define a relation $R_S$ on $\Omega$ as follows: For any $x, y\in \Omega$, $(x, y)\in R_S$ if and only if $x=y$ or there exists an $S$-path connecting $x$ and $y$.

(ii) We define a relation $R_W$ on $\Omega$ as follows: For any $x, y\in \Omega$, $(x, y)\in R_W$ if and only if $x=y$ or there exists an $W$-path connecting $x$ and $y$.
\end{Def}

It is direct to check that both $R_S$ and $R_W$ are equivalence relations.

\begin{Def} (Strong and weak nodal domain). \label{3.4}
Let $\Gamma=(H, \sigma)$ be a signed hypergraph where $H=(V, E)$ and $f: V\rightarrow \mathbb{R}$ be a function. Let $\Omega=\{v\in V: f(v)\neq 0\}$ be the support set of $f$.

(i) We denote by $\{S_i\}^p_{i=1}$ the equivalence classes of the relation $R_S$ on $\Omega$. We call the induced subhypergraph of each $S_i$ a strong nodal domain of the function $f$. We denote the number $p$ of strong nodal domains of $f$ by $\mathfrak{S}(f)$.

(ii) We denote by $\{W_i\}^q_{i=1}$ the equivalence classes of the relation $R_W$ on $\Omega$. We call the induced subhypergraph of each set

$W_i^0: = W_i\cup \{x\in V:$ there exists a $W$-path from $x$ to some vertex in $W_i\}$\\
a weak nodal domain of the function $f$. We denote the number $q$ of weak nodal domains of $f$ by $\mathfrak{W}(f)$.
\end{Def}

Notice that $W_i^0$ is obtained from $W_i$ by absorbing the zeros around it.

Next, we illustrate above definitions  by an example.

{\bf Example 1} We consider the signed hypergraph $\Gamma=(H, \sigma)$ given in Figure 1 and the normalized Laplacian operator
\[
L=\left(
  \begin{array}{ccccccccc}
  1 & -\frac{1}{3} & -\frac{1}{3} & 0 & -\frac{1}{3} & -\frac{1}{3} & \frac{1}{3} & 0 & 0 \\
  -1 & 1 & -1 & 0 & 0 & 0 & 0 & 0 & 0 \\
  -\frac{1}{3} & -\frac{1}{3} & 1 & -\frac{1}{3} & -\frac{1}{3} & 0 & 0 & 0 & \frac{1}{3} \\
  0 & 0 & -1 & 1 & -1 & 0 & 0 & 0 & 0 \\
  0 & 0 & 0 & -\frac{1}{3} & 1 & -\frac{1}{3} & 0 & \frac{1}{3} & 0 \\
  -1 & 0 & 0 & 0 & -1 & 1 & 0 & 0 & 0 \\
  1 & 0 & 0 & 0 & 0 & 0 & 1 & 0 & 0 \\
  0 & 0 & 0 & 0 & 1 & 0 & 0 & 1 & 0 \\
  0 & 0 & 1 & 0 & 0 & 0 & 0 & 0 & 1 \\
  \end{array}
\right).
\]

By numerical computation, we obtain the eigenvalues of $L$ listed below:
\begin{eqnarray*}
  \lambda_1\approx -0.51 \leq \lambda_2\approx 0.22 \leq \lambda_3\approx 0.33 \leq \lambda_4=\lambda_5=\lambda_6=1 \leq \lambda_7\approx 1.95 \leq \lambda_8=\lambda_9=2
\end{eqnarray*}

\begin{figure}[!hbpt]
\begin{center}
\includegraphics[scale=0.3]{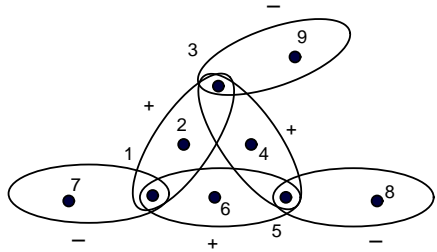}\\
\caption{$\Gamma=(H, \sigma)$}
\end{center}
\end{figure}

The following is a system of corresponding eigenfunctions of $L:$
\begin{eqnarray*}
f_1 &\approx& (-0.38, -0.50, -0.38, -0.38, -0.2, -0.38, 0.25, 0.13, 0.25)^{T}.\\
f_2 &\approx& (-0.22, -0.56, -0.22, 0.19, 0.37, 0.19, 0.28, -0.47, 0.28)^{T}.\\
f_3 &\approx& (0.3, 0, -0.3, -0.45, 0, 0.45, -0.45, 0, 0.45)^{T}.\\
f_4 &\approx& (0, -0.38, 0, 0.47, 0, -0.33, -0.71, 0.14, 0.08)^{T}.\\
f_5 &\approx& (0, -0.63, 0, 0.18, 0, 0.18, -0.44, 0.36, -0.45)^{T}.\\
f_6 &\approx& (0, 0.4, 0, 0.37, 0, -0.3, 0.1, 0.07, 0.77)^{T}.\\
f_7 &\approx& (-0.07, 0.16, -0.07, -0.45, 0.5, -0.45, -0.08, 0.53, -0.08)^{T}.\\
f_8 &\approx& (0.09, 0, -0.09, -0.4, 0.49, -0.58, 0.09, 0.49, -0.09)^{T}.\\
f_9 &\approx& (0.23, 0, -0.23, 0.64, -0.41, 0.18, 0.23, -0.41, -0.23)^{T}.
\end{eqnarray*}

We list the strong and weak nodal domains of each eigenfunction in Table 1. Notice that we only provide vertex subsets. The strong and weak nodal domains are the induced subhypergraphs of those vertex subsets in Table 1.
\begin{table}
\resizebox{\textwidth}{!}{
\begin{tabular}{|c|c|c|}
  \hline
  Eigenfunction & Strong nodal domain & Weak nodal domain \\
  \hline
  $f_1$ & \{1, 2, 3, 4, 5, 6, 7, 8, 9\}  & \{1, 2, 3, 4, 5, 6, 7, 8, 9\} \\
  \hline
  $f_2$ & \{1, 2, 3, 7, 9\}, & \{1, 2, 3, 7, 9\}, \\
        & \{4, 5, 6, 8\} & \{4, 5, 6, 8\} \\
  \hline
  $f_3$ & \{1, 6, 7\}, & \{1, 2, 6, 7\}, \\
        & \{3, 4, 9\} & \{3, 4, 5, 8, 9\} \\
  \hline
  $f_4$ & \{2\}, \{4\}, \{6\}, \{7\}, \{8\}, \{9\}, & \{2, 3, 9\}, \{5, 6, 8\}, \{4\}, \{1, 7\}  \\
  \hline
  $f_5$ & \{2\}, \{4\}, \{6\}, \{7\}, \{8\}, \{9\},  & \{1, 6, 7\}, \{3, 4, 9\}, \{2\}, \{5, 8\} \\
  \hline
  $f_6$ & \{2\}, \{4\}, \{6\}, \{7\}, \{8\}, \{9\}, & \{1, 5, 6, 7, 8\}, \{2, 3, 4\}, \{9\} \\
  \hline
  $f_7$ & \{2\}, \{5\}, \{7\}, \{8\},   & \{2\}, \{5\}, \{7\}, \{8\}, \\
        & \{9\}, \{1, 3, 4, 6\}  & \{9\}, \{1, 3, 4, 6\}  \\
  \hline
  $f_8$ & \{1, 5, 8\}, \{3, 4\}, \{6\},  & \{1, 5, 8\}, \{2, 3, 4\}, \{6\},  \\
        & \{7\}, \{9\}   & \{7\}, \{9\} \\
  \hline
  $f_9$ & \{1, 6\}, \{3, 5\}, \{4\}, & \{1, 6\}, \{2, 3, 5\}, \{4\}, \\
        &  \{7\}, \{8\}, \{9\}  & \{7\}, \{8\}, \{9\}  \\
  \hline
\end{tabular}
}
\caption{Strong and weak nodal domains}
\end{table}
\subsection{Basic properties of weak nodal domains}
We say two domains $D_i$ and $D_j$ are adjacent, denoted by $D_i\sim D_j$, if there exist $x\in D_i$ and $y\in D_j$ such that $x, y\in e\in E$. By definition, we have the following proposition.

\begin{pro}\label{3.7}
Let $\{D_i\}^q_{i=1}$ be the weak nodal domains of a non-zero function $f$ on a signed hypergraph $\Gamma=(H, \sigma)$ where $H=(V, E)$. Let $H_D=(V_D, E_D)$ be the graph defined by
\begin{equation*}
  V_D : =\{D_i\}^q_{i=1}, \ and \ E_D : = \{\{D_i, D_j\} : D_i\sim D_j\}.
\end{equation*}
If the hypergraph $H$ is connected, so does $H_D$.
\end{pro}

\begin{proof}
Let $D$ and $D^{\prime}$ be any two weak nodal domains. Choose two vertices $x$ and $x^{\prime}$ such that $x\in D$ and $x^{\prime}\in D^{\prime}$. Since hypergraph $H$ is connected, there exists a path $x=x_0e_1x_1e_2x_2\cdots e_mx_m=x^{\prime}$ connecting $x$ and $x^{\prime}$. Set $i_0 : =max\{i : x_i\in D\}$. Then we have $f(v_i)\neq 0, i=1, 2, \cdots, l-2$ and $f(x_{i_0+1})\neq 0$ for $e_{x_{i_0+1}}=\{x_{i_0}, v_1, \cdots, v_{l-2}, x_{i_0}+1\}$. Therefore, $x_{i_0+1}$ and $x$ belongs to different equivalent classes of the relation $R_W$, i.e., $x_{i_0+1}$ lies in a weak nodal domain $D_1\sim D$. Applying the above argument iteratively, we find a path $D\sim D_1\sim \cdots \sim D^{\prime}$ from $D$ to $D^{\prime}$ in the graph $H_D$. That is, the graph $H_D$ is connected.
\end{proof}

For any zero vertex, we have the following observation.
\begin{pro}\label{3.8}
Let f be a non-zero function on a signed hypergraph $\Gamma=(H, \sigma)$. Then for any three weak nodal domain $D_1, D_2, D_3$ of $f$ we have $D_1\cap D_2\cap D_3=\emptyset$.
\end{pro}

\begin{proof}
Suppose that $D_1\cap D_2\cap D_3\neq \emptyset$. Let $x\in D_1\cap D_2\cap D_3$. Then we have $f(x)=0$. By definition, we can find for each $i\in \{1, 2, 3\}$, there is an edge $e_i$ such that $x_0^i, x\in e_i$ and $f(x_0^i)\neq 0$ in $D_i$. Set
\begin{equation*}
  a_i := f(x^i_0)sgn(e_i)\in \mathbb{R}, i = 1, 2, 3.
\end{equation*}
Since $D_1, D_2, D_3$ are different from each other, we obtain
\begin{eqnarray*}
  f(x^1_0)sgn(e_1)sgn(e_2)f(x^2_0)&<& 0 \\
  f(x^1_0)sgn(e_1)sgn(e_3)f(x^3_0)&<& 0 \\
  f(x^2_0)sgn(e_2)sgn(e_3)f(x^3_0)&<& 0.
\end{eqnarray*}
That is
\begin{equation*}
  a_1a_2<0, a_1a_3<0, a_2a_3<0,
\end{equation*}
which is a contradiction.
\end{proof}

\begin{corollary}\label{3.9}
Let $f$ be a non-zero function on a signed hypergraph $\Gamma=(H, \sigma)$. Let $x$ be a vertex lying in two weak nodal domains $D$ and $D^{\prime}$. Then the set
\begin{equation*}
  B(x) : = \{y\in V : \ there \ exist \ an \ edge \ such \ that \ \{x, y\}\subseteq e_i \ and \ f(x)=0.\}
\end{equation*}
contains in $D\cup D^{\prime}$. In particular, we have $S_1(x)\subset D\cup D^{\prime}$ where $S_1(x)=\{y\in V : y\sim x\}$.
\end{corollary}

\begin{proof}
Let $y\in B(x)$. If $f(y)=0$, then we have $y\in D\cup D^{\prime}$ by definition. In the case of $f(y)\neq 0$, we suppose that $y$ lies in a weak nodal domain $D$ different from $D$ and $D^{\prime}$. Then the vertex $x\in D\cap D^{\prime} \cap \bar{D}$, which is a contradiction by Proposition \ref{3.8}.
\end{proof}

\section{Strong and weak nodal domains on signed hypergraphs}
In this section, we prove the following nodal domain theorem. The proof is a neat extension of methods from \cite{Davis2001} and \cite{Ge2022}.

\begin{theorem}\label{4.1}
Let $\Gamma=(H, \sigma)$ be a signed hypergraph where $H=(V, E)$, $L$ is the normalized Laplacian operator associated to $\Gamma$ and $\lambda_k$ be its $k$-th eigenvalue. For any eigenfunction $f_k$ corresponding to $\lambda_k$, i.e., $Lf_k=\lambda_kf_k$, we have
\begin{equation*}
  \mathfrak{S}(f_k)\leq k+r-1, \ and \ \mathfrak{W}(f_k)\leq k+c-1,
\end{equation*}
where $r$ is the multiplicity of $\lambda_k$ and $c$ is the number of connected components of $H$. In
particular, when the hypergraph $H$ is connected, we have $\mathfrak{W}(f_k)\leq k$.
\end{theorem}

We prepare a crucial lemma, which is a reformulation of Duval and Reiner \cite{Duval1999}.

\begin{lemma}\label{4.2}
Let $\Gamma=(H, \sigma)$ be a signed hypergraph where $H=(V, E)$, $L$ is the normalized Laplacian operator associated to $\Gamma$. Then for any two functions $f, g: V\rightarrow \mathbb{R}$ we have
\begin{equation*}
  <fg, L(fg)>=<fg, fLg>+\sum_{x, y\in e\in E} deg(x)(-L_{xy})g(x)g(y)(f(x)-f(y))^2
\end{equation*}
where $x, y\in e\in E$.
\end{lemma}

\begin{proof}
By a direct calculation, we have
\begin{eqnarray*}
  &&<fg, L(fg)> \\
  &=& \sum_{x\in V} deg(x)f(x)g(x)\big[ -\frac{1}{deg(x)}\sum_{x, y\in e\in E}
  A_{xy}f(y)g(y)+f(x)g(x) \big] \\
  &=& \sum_{x\in V} deg(x)f(x)g(x)\big[ -\frac{1}{deg(x)}\sum_{x, y\in e\in E}
  (A_{xy}f(y)g(y)-A_{xy}f(x)g(y)\\
  &&+A_{xy}f(x)g(y))+f(x)g(x) \big]\\
  &=& \sum_{x\in V} deg(x)f(x)g(x) \sum_{x, y\in e\in E}
  L_{xy}g(y)(f(y)-f(x))\\
  && +\sum_{x\in V} deg(x)f(x)g(x)\big[ -\frac{1}{deg(x)}\sum_{x, y\in e\in E}
  A_{xy}f(x)g(y)+f(x)g(x) \big]\\
  &=& \sum_{x\in V} deg(x)f(x)g(x) \sum_{x, y\in e\in E}L_{xy}g(y)(f(x)-f(y))\\
  && +\sum_{x\in V} deg(x)f^2(x)g(x)\big[ -\frac{1}{deg(x)}\sum_{x, y\in e\in E}
  A_{xy}g(y)+g(x) \big]\\
  &=& \sum_{x\in V} deg(x)f(x)g(x) \sum_{x, y\in e\in E}L_{xy}g(y)(f(x)-f(y))+<fg, fLg>\\
  &=& \sum_{x, y\in e\in E} deg(x)(-L_{xy})g(x)g(y)(f(x)-f(y))^2+<fg, fLg>
\end{eqnarray*}
This completes the proof.
\end{proof}

The following corollary will be crucial for the proof of Theorem \ref{5.6} in Section 5.
\begin{corollary}\label{4.3}
Let $g$ be an eigenfunction of $L$ such that $Lg=\lambda_k g$. Let $D(g)$ be the diagonal matrix with $D(g)_{xx}=g(x)$. Then we have
\begin{equation*}
  <f, D(g)(L-\lambda_k I)D(g)f>=\sum_{x, y\in e\in E} deg(x)(-L_{xy})g(x)g(y)(f(x)-f(y))^2
\end{equation*}
\end{corollary}

\begin{proof}
By Lemma \ref{4.2}, we have
\begin{eqnarray*}
  <fg, (L-\lambda_k I)fg>&=&<fg, L(fg)-fLg> \\
   &=& \sum_{x, y\in e\in E} deg(x)(-L_{xy})g(x)g(y)(f(x)-f(y))^2.
\end{eqnarray*}
Then the corollary follows directly by observing that
\begin{eqnarray*}
  <fg, (L-\lambda_k I)fg>&=&<f, D(g)(L-\lambda_k I)D(g)f>.
\end{eqnarray*}
\end{proof}

For any $k$, let $f_k$ be the eigenfunction of $L$ corresponding to the $k$-th eigenvalue $\lambda_k$. Next, we show the estimates of $\mathfrak{S}(f_k)$ and $\mathfrak{W}(f_k)$ in Theorem 4.1.

\begin{proof}(Proof of Theorem 4.1:)
Estimate of $\mathfrak{S}(f_k)$. Let $\{\Omega_i\}^m_{i=1}$ be the strong nodal domains of $f_k$, where $m=\mathfrak{S}(f_k)$. For each $i$, we define
\begin{equation}\label{}
     g_i(x)=\left\{
      \begin{array}{ll}
      f_k(x), \ if \ \ x \in \Omega_i;\\
      0, \ \ \ \ \ \ otherwise.
   \end {array}
   \right.
\end{equation}
Since $g_i, i=1, 2, \cdots, m$, are linearly independent, we can find $a_i\in \mathbb{R}, i=1, 2, \cdots, m$, such
that the function $g: =\sum^m_{i=1} a_ig_i$ satisfies
\begin{equation*}
  <g, f_i>=0, \ for \ i=1, 2, \cdots, m-1.
\end{equation*}
We introduce a function $a: V\rightarrow \mathbb{R}$ defined as $a(x)=a_i$ if $x\in \Omega_i$ for some $i$ and $a(x) = 0$ otherwise. Then we can write $g=af_k$. Applying Lemma \ref{4.2} yields
\begin{eqnarray*}
<g, Lg>&=& <af_k, aLf_k>\\
&& +\sum_{x_j,x_{j+1}\in e\in E} deg(x_j) (-L_{x_jx_{j+1}})f_k(x_j)f_k(x_{j+1})(a(x_j)-a(x_{j+1}))^2\\
&=& \lambda_k<g, g>\\
&& +\sum_{x_j,x_{j+1}\in e\in E} deg(x_j) (-L_{x_jx_{j+1}})f_k(x_j)f_k(x_{j+1})(a(x_j)-a(x_{j+1}))^2
\end{eqnarray*}
where $x_j, x_{j+1}\in e=\{x_i, x_{i+1}, \cdots, x_{i+l-1}\} \in E$ for each $j=i, i+1, \cdots, i+l-2$.

By Theorem \ref{2.9}, we derive
\begin{equation*}
  \lambda_m \leq \frac{<g, Lg>}{<g, g>}\leq \lambda_k+\sum_{x_j,x_{j+1}\in e\in E} deg(x_j) (-L_{x_jx_{j+1}})f_k(x_j)f_k(x_{j+1})(a(x_j)-a(x_{j+1}))^2.
\end{equation*}
For each edge $e$, $x_j, x_{j+1}\in e=\{x_i, x_{i+1}, \cdots, x_{i+l-1}\} \in E$ for each $j=i, i+1, \cdots, i+l-2$, if $(-L_{x_jx_{j+1}})f_k(x_j)f_k(x_{j+1})> 0$, then $f_k(x_j)f_k(x_{j+1})sgn(e)> 0$. That is, the vertices $x_j$ and $x_{j+1}$ lie in the same strong nodal domain.

Hence, $a(x_j)-a(x_{j+1})= 0$ and $(-L_{x_jx_{j+1}})f_k(x_j)f_k(x_{j+1})(a(x_j)-a(x_{j+1}))^2=0$. If, otherwise, $(-L_{x_jx_{j+1}})f_k(x_j)f_k(x_{j+1})\leq 0$, we have $(-L_{x_jx_{j+1}})f_k(x_j)f_k(x_{j+1})(a(x_j)-a(x_{j+1}))^2\leq 0$. Therefore, we obtain
\begin{equation*}
  \sum_{x_j,x_{j+1}\in e\in E} deg(x_j)  (-L_{x_jx_{j+1}})f_k(x_j)f_k(x_{j+1})(a(x_j)-a(x_{j+1}))^2 \leq 0.
\end{equation*}
This leads to $\lambda_m \leq \lambda_k$. Recall that $\lambda_k<\lambda_{k+r}$, we have $\mathfrak{S}(f_k)=m\leq k+r-1$.
\end{proof}

In order to show the estimate of $\mathfrak{W}(f_k)$, we first prepare the following discrete unique continuation lemma for eigenfunctions.

\begin{lemma}\label{4.4}
Let $\Gamma=(H, \sigma)$ be a connected signed hypergraph where $H=(V, E)$, $L$ is the normalized Laplacian operator associated to $\Gamma$. Consider an eigenfunction $f$ of the normalized Laplacian operator $L$ corresponding to an eigenvalue $\lambda$. Define a function $g_i$ as below
\begin{equation}\label{3}
     g_i(x)=\left\{
      \begin{array}{ll}
      f(x), \ \ if \ \ x \in D_i;\\
      0, \ \ \ \ \ \ otherwise.
   \end {array}
   \right.
\end{equation}
If the function
\begin{equation*}
  g :=\sum^m_{i=1}a_ig_i, \ where \ a_i\in \mathbb{R}, i=1,2, \cdots, m,
\end{equation*}
is also an eigenfunction of $L$ corresponding to $\lambda$, then we have
\begin{equation*}
  a_1=a_2=\cdots=a_m.
\end{equation*}
\end{lemma}

\begin{proof}
We first observe that $g$ can be reformulated as a product of two functions, $g=af$, where the function $a : V\rightarrow \mathbb{R}$ is defined as $a(x)=a_i$ if $x\in D_i$ and $f(x)\neq 0$, and $a(x)=0$ otherwise.

For any adjacent $D_i$ and $D_j$, we can always find $x_{j^{\prime}}\in D_i, x_{j^{\prime}+1}\in D_j\setminus D_i$ such that $x_{j^{\prime}}, x_{j^{\prime}+1}\in e^{\prime}=\{x_{i^{\prime}}, x_{(i+1)^{\prime}}, \cdots, x_{(i+l-1)^{\prime}}\} \in E$ for each $j^{\prime}=i^{\prime}, (i+1)^{\prime}, \cdots, (i+l-2)^{\prime}$. Indeed, when $D_i\cap D_j\neq \emptyset$, due to the connectedness of $D_j$, there exists a path connecting any $x_i\in D_i\cap D_j$ and any $z\in D_j \setminus D_i$.

Next, we show $a_i=a_j$ for any two adjacent weak nodal domains $D_i$ and $D_j$ if $a_i\neq 0$. We divide our arguments into two cases.

{\bf Case \ 1:} $f(x_{j^{\prime}})\neq 0$. By Lemma \ref{4.2}, we have
\begin{equation}\label{4}
  \lambda \leq \frac{<g, Lg>}{<g, g>}\leq \lambda+\frac{1}{<g, g>}\sum_{x_j,x_{j+1}\in e\in E} deg(x_j) (-L_{x_jx_{j+1}})f(x_j)f(x_{j+1})(a(x_j)-a(x_{j+1}))^2
\end{equation}
where $x_j, x_{j+1}\in e=\{x_i, x_{i+1}, \cdots, x_{i+l-1}\} \in E$ for each $j=i, i+1, \cdots, i+l-2$.

For each edge $e$, where $x_j, x_{j+1}\in e=\{x_i, x_{i+1}, \cdots, x_{i+l-1}\} \in E$ for each $j=i, i+1, \cdots, i+l-2$, if $(-L_{x_jx_{j+1}})f(x_j)f(x_{j+1})> 0$, then $f(x_j)f(x_{j+1})sgn(e)> 0$. That is, the vertices $x_j$ and $x_{j+1}$ lie in the same weak nodal domain with both $f(x_j)$ and $f(x_{j+1})$ nonzero, which means $a(x_j)-a(x_{j+1})= 0$. Therefore, we obtain
\begin{equation}\label{5}
(-L_{x_jx_{j+1}})f(x_j)f(x_{j+1})(a(x_j)-a(x_{j+1}))^2\leq 0,
\end{equation}
where $x_j, x_{j+1}\in e=\{x_i, x_{i+1}, \cdots, x_{i+l-1}\} \in E$ for each $j=i, i+1, \cdots, i+l-2$.

Combining (4) and (5) yields that for any $e\in E$,
\begin{equation}\label{6}
(-L_{x_jx_{j+1}})f(x_j)f(x_{j+1})(a(x_j)-a(x_{j+1}))^2=0,
\end{equation}
where $x_j, x_{j+1}\in e=\{x_i, x_{i+1}, \cdots, x_{i+l-1}\} \in E$ for each $j=i, i+1, \cdots, i+l-2$.

For the edge $e^{\prime}\in E$, where $x_{j^{\prime}}, x_{j^{\prime}+1}\in e^{\prime}=\{x_{i^{\prime}}, x_{(i+1)^{\prime}}, \cdots, x_{(i+l-1)^{\prime}}\} \in E$ for each $j^{\prime}=i^{\prime}, (i+1)^{\prime}, \cdots, (i+l-2)^{\prime}$, we have $f(x_{j^{\prime}+1})\neq 0$ since $x_{j^{\prime}+1}\in D_i\cap D_j$. Moreover, we have
\begin{equation*}
  (-L_{x_{j^{\prime}}x_{j^{\prime}+1}})f(x_{j^{\prime}})f(x_{j^{\prime}+1})<0,
\end{equation*}
since $x_{j^{\prime}}, x_{j^{\prime}+1}$ lies in two different weak nodal domains with both $f(x_{j^{\prime}})$ and $f(x_{j^{\prime}+1})$ nonzero. By (6), we have $a(x_{j^{\prime}})=a(x_{j^{\prime}+1})$. Since $f(x_{j^{\prime}})\neq 0$ and $f(x_{j^{\prime}+1})\neq 0$, we have $a_i=a_j$.

{\bf Case \ 2:} $f(x_{j^{\prime}})=0$. By the Corollary \ref{3.9}, we have $S_1(x_{j^{\prime}}) : =\{x_{j^{\prime}+1} \in V : x_{j^{\prime}}, x_{j^{\prime}+1}\in e^{\prime}=\{x_{i^{\prime}}, x_{(i+1)^{\prime}}, \cdots, x_{(i+l-1)^{\prime}}\} \in E \}\subset D_i\cup D_j$ for each $j^{\prime}=i^{\prime}, (i+1)^{\prime}, \cdots, (i+l-2)^{\prime}$. We define a function $h : =f-\frac{1}{a_i}g$. Observe that $h|_{D_i}=0$, and $h$ is an eigenfunction of $L$ corresponding to $\lambda$. So we have
\begin{eqnarray*}
  0 &=& -\lambda h(x_{j^{\prime}})=Lh(x_{j^{\prime}})= \sum_{x_{j^{\prime}}, x_{j^{\prime}+1}\in e^{\prime}\in E} L_{x_{j^{\prime}}x_{j^{\prime}+1}} h(x_{j^{\prime}+1})\\
  &=& \sum_{x_{j^{\prime}}, x_{j^{\prime}+1}\in e^{\prime}\in E, x_{j^{\prime}+1}\in D_j \setminus D_i} L_{x_{j^{\prime}}x_{j^{\prime}+1}} h(x_{j^{\prime}+1})\\
  &=& (1-\frac{a_j}{a_i})\sum_{x_{j^{\prime}}, x_{j^{\prime}+1}\in e^{\prime}\in E, x_{j^{\prime}+1}\in D_j \setminus D_i} L_{x_{j^{\prime}}x_{j^{\prime}+1}} f(x_{j^{\prime}+1})
\end{eqnarray*}
for $x_{j^{\prime}+1} \in V$ where $x_{j^{\prime}}, x_{j^{\prime}+1}\in e^{\prime}=\{x_{i^{\prime}}, x_{(i+1)^{\prime}}, \cdots, x_{(i+l-1)^{\prime}}\} \in E$ for each $j^{\prime}=i^{\prime}, (i+1)^{\prime}, \cdots, (i+l-2)^{\prime}$.

Since there exists a vertex $x_{j^{\prime}+2}\in e^{\prime\prime}$ such that $f(x_{j^{\prime}+2})\neq 0$, where $x_{j^{\prime}}, x_{j^{\prime}+2}\in e^{\prime\prime}=\{x_{i^{\prime\prime}}, x_{(i+1)^{\prime\prime}}, \cdots, x_{(i+l-1)^{\prime\prime}}\} \in E$ for each $j^{\prime\prime}=i^{\prime\prime}, (i+1)^{\prime\prime}, \cdots, (i+l-2)^{\prime\prime}$.

By the definition of the weak nodal domain, it holds that
\begin{equation*}
  f(x_{j^{\prime}+1})sgn(e^{\prime})sgn(e^{\prime\prime})f(x_{j^{\prime}+2})>0
\end{equation*}
for any $x_{j^{\prime}+1}, x_{j^{\prime}+2} \in D_j\setminus D_i$.

Thus we have
\begin{equation*}
  f(x_{j^{\prime}+1}) L_{x_{j^{\prime}}x_{j^{\prime}+1}} L_{x_{j^{\prime}}x_{j^{\prime}+2}} f(x_{j^{\prime}+2})>0
\end{equation*}
for any $x_{j^{\prime}+1}, x_{j^{\prime}+2} \in D_j\setminus D_i$.

Therefore, we have
\begin{equation*}
  \sum_{x_{j^{\prime}}, x_{j^{\prime}+1}\in e^{\prime}\in E, x_{j^{\prime}+1}\in D_j \setminus D_i} L_{x_{j^{\prime}}x_{j^{\prime}+1}} f(x_{j^{\prime}+1})\neq 0
\end{equation*}
This tells that $a_i=a_j$.

Since $g$ is an eigenfunction, at least one of $a_i, i=1, \cdots, m$ is non-zero. Then the lemma follows directly from the above argument and the connectedness from Proposition 3.7.
\end{proof}

\begin{proof}(Proof of Theorem \ref{4.1}:)
Estimate of $\mathfrak{W}(f_k)$. We first assume that the signed hypergraph $\Gamma=(H, \sigma)$ is connected. We denote all weak nodal domains of $f_k$ by $\{D_i\}^m_{i=1}$, where $m=\mathfrak{W}(f_k)$. We introduce for each $i$
\[
     g_i(x)=\left\{
      \begin{array}{ll}
      f_k(x), \ if \ \ x \in D_i;\\
      0, \ \ \ \ \ \ otherwise.
   \end {array}
   \right.
\]
Let $a_i\in \mathbb{R}, i=1, 2, \cdots, m$, be $m$ constants such that $g: =\sum^m_{i=1} a_ig_i$ satisfies $<g, f_i>=0$, for $i=1, 2, \cdots, m-1$. We define a function $a: V\rightarrow \mathbb{R}$ as $a(x)=a_i$ if $x\in D_i$ and $f_k(x)\neq 0$, and $a(x)=0$ otherwise. By construction, $g=af_k$. We then derive
\begin{equation}\label{7}
\lambda_m \leq \frac{<g, Lg>}{<g, g>}\leq \lambda_k+\sum_{x_j,x_{j+1}\in e\in E}deg(x_j) (-L_{x_jx_{j+1}})f_k(x_j)f_k(x_{j+1})(a(x_j)-a(x_{j+1}))^2.
\end{equation}
where $x_j, x_{j+1}\in e=\{x_i, x_{i+1}, \cdots, x_{i+l-1}\} \in E$ for each $j=i, i+1, \cdots, i+l-2$.

Similarly as in the proof of (5), we have
\begin{equation}\label{8}
(-L_{x_jx_{j+1}})f(x_j)f(x_{j+1})(a(x_j)-a(x_{j+1}))^2\leq 0,
\end{equation}
where $x_j, x_{j+1}\in e=\{x_i, x_{i+1}, \cdots, x_{i+l-1}\} \in E$ for each $j=i, i+1, \cdots, i+l-2$.\\
Hence,
\begin{equation}\label{9}
\lambda_m \leq \lambda_k.
\end{equation}
We argue by contradiction. Suppose $m>k$, then $\lambda_m \geq \lambda_k$. By (7) and (9), we have
\begin{equation}\label{10}
\lambda_m=\frac{<g, Lg>}{<g, g>}=\lambda_k.
\end{equation}
Hence, $g$ is an eigenfunction of $L$ corresponding $\lambda_k$.

Then we can apply Lemma \ref{4.4} to show that $g=af_k$ where $a$ is a nonzero constant function. However, $<g, f_j>=0$ for any $j<m$ by construction. Our assumption $m>k$ then implies $<g, f_k>=0$. That is, $a<f_k, f_k>=0$. This is a contradiction. So we get
\begin{equation}\label{11}
\mathfrak{W}(f_k)=m\leq k.
\end{equation}

In general, we denote by $\{\Gamma_i\}^c_{i=1}$ the $c$ connected components of the signed hypergraph $\Gamma=(H, \sigma)$ where $H=(V, E)$, $L$ is the normalized Laplacian operator associated to $\Gamma$. Let $L^i, f_k^i$ be the restriction of $L, f_k$ to the connected component $\Gamma_i$ respectively. Then either $f_k^i$ is identically zero on $\Gamma_i$ or $L_if_k^i=\lambda^i_{k_i}f_k^i$, where $\lambda^i_{k_i}=\lambda_k$ is the $k_i$-th eigenvalue of $L^i$. Moreover, we can assume $\lambda^i_{k_{i-1}}<\lambda^i_{k_i}$. Without loss of generality, we assume $\{\Gamma_i\}^c_{l=1}$ be the connected components on which $f_k$ is not identically zero. Employing the fact (11) we estimate
\begin{equation*}
  \mathfrak{W}(f_k)\leq \sum_{i=1}^l k_i \leq \sum_{i=1}^l (k_i-1)+l <k+l\leq k+c.
\end{equation*}
This completes the proof.
\end{proof}

\section{Lower bound of the number of strong nodal domains}

\noindent

 In this section we  give a lower bound estimates of the number of strong nodal domains.

First, We give some notations. A cycle $\{x_i\}^n_{i=1}$ of a signed hypergraph $\Gamma=(H, \sigma)$ is called a strong nodal domain cycle ($S$-cycle for short) of a function $f$ on $\Gamma$ if it is an $S$-path of $f$ and $x_n=x_1$.

\begin{Def}\label{5.1}
Let $H=(V(H), E(H))$ be a hypergraph. Define
\begin{equation*}
  l(H): =\sum_{e\in E(H)}(|e|-1)-|V(H)|+c(H),
\end{equation*}
where $c(H)$ is the number of connected components of $H$. Let $\Gamma=(H, \sigma)$ be a signed hypergraph and $f: V\rightarrow \mathbb{R}$ be a function. Let $H_1$ be a hypergraph whose vertex set $V(H_1)=V$ and edge set
$E(H_1) : =\{e : e=\{x_i, x_{i+1}, \cdots, x_{i+l-1}\} \in E, f(x_j)sgn(e)f(x_{j+1})>0 \ for \ each \ j=i, i+1, \cdots, i+l-2.\}$ Define
\begin{equation*}
  l_{+}(H, \sigma, f): =\sum_{e\in E(H_1)}(|e|-1)-|V(H_1)|+c(H_1),
\end{equation*}
where $c(H_1)$ is the number of connected components of $H_1$.
\end{Def}

{\bf Remark.} The number $l(H)$, is the minimal number of edges that need to be removed from $H$ in order to turn it into a hyperforest. $l(H)$ is cyclomatic numbers of hypergraph $H$ in \cite{Acharya1979}. The number of $l_{+}(H, \sigma, f)$ is the dimension of the vector space of $S$-cycles of the function $f$ on the signed hypergraph $\Gamma=(H, \sigma)$.

\begin{Def}\label{5.2}
Let $H=(V, E)$ be a hypergraph. A vertex $x\in V$ is called a tree-like vertex if the induced subhypergraph by removing $x$ from $H$ increases the number of connected components by $d_{x}-1$.
\end{Def}

All vertices in a hyperforest are tree-like. Moreover, we have the observations below.

\begin{pro}\label{5.3}
The tree-like vertices have the following properties:

(i) Let $x$ be a tree-like vertex in a hypergraph $H=(V, E)$. Let $H\setminus \{y\}$ be a hypergraph obtained from $H$ by removing any vertex $y\neq x$ and its incident edges. Then $x$ is still a tree-like vertex in the graph $H\setminus \{y\}$.

(ii) Let $H^{\prime}$ be the induced subhypergraph of $H=(V, E)$ on the set $V\setminus Y$ where $Y$ is a set of tree-like vertices. Then we have $l(H^{\prime})=l(H)+\sum_{y\in Y} d_y$.
\end{pro}

\begin{proof}
The property (i) follows directly from the observation that a vertex is tree-like if and only if it belongs to no cycle. Moreover, for any $y\in Y$, we have
\begin{equation*}
  l(H\setminus \{y\})=\sum_{e\in E}(|e|-1)-(|V|-1)+(c(H)+d_y-1)=l(H)+d_y.
\end{equation*}
By (i), we can apply the above argument iteratively to conclude (ii).
\end{proof}

Next, we define the following particular set of zeros of a function on a hypergraph.

\begin{Def}\label{5.4}
Let $H=(V, E)$ be a hypergraph and $f: V\rightarrow \mathbb{R}$ be a function. We define the Fiedler zero set of $f$ on $H$ as below:
\begin{eqnarray*}
  \mathcal{F}(H, f) &:=& \{x\in V : f(x)=0, \ and \ either \ f(y)=0 \ for \ all \ y \ such \ that \ x, y\in e\in E, \\
  && or \ x \ is \ not \ tree-like\}.
\end{eqnarray*}
We denote by $\mathcal{F}^{c}(H, f)$ the complement of $\mathcal{F}(H, f)$ in the zero set of $f$, i.e.,
\begin{eqnarray*}
  \mathcal{F}^{c}(H, f) &:=& \{x\in V : f(x)=0, \ x \ is \ tree-like \ and \ there \ exists \ y \ such \ that  \\
   &&  \ x, y\in e\in E \ and \ f(y)\neq 0\}.
\end{eqnarray*}
\end{Def}

The following definition of weak vertex addition of $\hat{v}$ is similar with the induced subhypergraph by removing $\hat{v}$ from $\Gamma=(V, H)$.

\begin{Def}[\cite{Reff2014}]\label{5.5}
Given $\hat{v}\in V$, we let $\Gamma-\hat{v} : =(\hat{V}, \hat{H})$, where:

$\bullet \hat{V}= V\setminus \{\hat{v}\}$, and

$\bullet \hat{H}=\{h \setminus \{\hat{v}\} : h\in H\}.$

$\Gamma-\hat{v}$ is obtained from $\Gamma$ by a weak vertex deletion of $\hat{v}$. $\Gamma$ is obtained from $\Gamma-\hat{v}$ by a weak vertex addition of $\hat{v}$. We also allow empty hyperedges.
\end{Def}

For the readers' convenience, we recall the definition of supertree.
\begin{Def}
A supertree is a hypergraph which is both connected and acyclic.
\end{Def}
A characterization of acyclic hypergraph has been given in Berge's (1976) textbook and particularly for the connected case is the following result.

\begin{pro}[\cite{Berge1976}]
If $H$ is a connected hypergraph with $n$ vertices and $m$ edges, then it is acyclic if and only if $\sum_{i\in [m]}(|e_i|-1)=n-1$.
\end{pro}

\begin{theorem}\label{5.6}
Let $\Gamma=(H, \sigma)$ be a signed hypergraph where $H=(V, E)$, $L$ is the normalized Laplacian operator associated to $\Gamma$. Let $\lambda_k$ be the $k$-th eigenvalue of $L$ with multiplicity $r$ and eigenfunction $f_k$ such that
\begin{equation*}
  \lambda_1 \leq \cdots \leq \lambda_{k-1} \leq \lambda_{k}=\cdots=\lambda_{k+r-1} < \lambda_{k+r} \leq \cdots \leq \lambda_{n}.
\end{equation*}
Then we have
\begin{equation*}
  \mathfrak{S}(f_k)\geq k+r-1-l^{\prime}+l_{+}-|\mathcal{F}|,
\end{equation*}
where $\mathcal{F}(H, f_k)$ is the Fielder zero set, $l_{+}=l_{+}(H, \sigma, f_k)$ is the dimension of S-cycle space of $H$, and $l^{\prime}=l(H^{\prime})$ is the dimension of the cycles space of $H^{\prime}$, where $H^{\prime}$ is the induced subhypergraph of $H$ on the set of nonzeros $V\setminus \{x : f_k(x)=0\}$.
\end{theorem}

Let us prepare five Lemmas. We first recall the following result of Fiedler.
\begin{lemma}[\cite{Fiedler1975}]\label{5.7}
Let
\[
A=\left(
  \begin{array}{cc}
    B & a \\
    a^{T} & b \\
  \end{array}
\right)
\]
be an $n\times n$ partitioned symmetric matrix where $b\in \mathbb{R}$. If there exists $u\in \mathbb{R}_{n-1}$ such that $Bu=0$ and $a^{T}\neq 0$. Then we have
\begin{equation*}
  p_{A}=p_{B}+1,
\end{equation*}
where $p_{A}$ and $p_{B}$ are the positive indices of inertia of $A$ and $B$, respectively.
\end{lemma}

\begin{lemma}\label{5.8}
Let $T=(V(T), E(T))$ be a supertree with $V(T)=\{1, 2, \cdots, n\}, E(T)=\{e_1, e_2, \cdots, e_m\}$. Then the $\sum_{i\in [m]}(|e_i|-1)$ linear forms $x_i-x_j$, where $i, j\in e \in E(T)$, and we choose two vertices for each edge $e\in E(T)$, are linearly independent.
\end{lemma}

\begin{proof}
Assume
\[\sum_{i,k\in e \in E(T)} a_{ik}(x_i-x_k)=0.\]
where not all $a_{ik}$ are equal to zero. Let $T_1$ be the subhypergraph of $T$ with the set of vertices $V(T_1)=\{1, \cdots, n\}$ and the set of edges $E(T_1)=\{e : i, k\in e \ and \ a_{ik}\neq 0\}$.
Since a supertree is a connected hypergraph not containing any cycle, then $T_1$ has at least one edge and each of its component is a supertree. We also know that every supertree with $n\geq 2$ vertices has at least one end-vertex (i.e. a vertex adjacent to a single edge), thus the component with at least one edge in $T_1$ has at least one end-vertex.  If this vertex is $n$ and the single adjacent edge $e_1$ such that $n-1, n\in e_1$, then $x_n$ is contained in the sum above in the only term $a_{n-1,n}(x_{n-1}-x_{n})$, where $a_{n-1,n}=0$. in contradiction with the assumption that this sum is zero.
\end{proof}

\begin{lemma}[\cite{Mulas2021}]\label{5.9}
If $\hat{\Gamma}$ is obtained from $\Gamma$ by weak-deleting $r$ vertices, then
\begin{equation*}
  \lambda_k(\Gamma) \leq \lambda_k(\hat{\Gamma}) \leq \lambda_{k+r}(\Gamma) \ for \ all \ k\in \{1, \cdots, n-r\}.
\end{equation*}
\end{lemma}

\begin{lemma}\label{5.10}
Consider a quadratic form
\begin{equation*}
  B=\sum^n_{i,j=1} a_{ij} (x_i-x_j)^2, \ where \ a_{ji}=a_{ij}\in \mathbb{R}.
\end{equation*}
Let $H=(V(H), E(H))$ be the hypergraph with $V=\{1,2,\cdots, n\}$ and $E=\{e: i,j \in e, a_{ij}\neq 0 \ and \ i\neq j\}.$ Let $H^{\prime}=(V, E(H^{\prime}))$ with $E(H^{\prime})=\{e: i,j \in e, a_{ij}> 0\}$ be a subhypergraph of $H$. For any spanning hyperforest $T$ of $H^{\prime}$, we have
\begin{equation}\label{6.1}
  p\leq \sum_{e\in E(T)}(|e|-1) \leq \sum_{e\in E(H^{\prime})}(|e|-1) \leq p+l,
\end{equation}
where $p$ is the positive index of inertia of $B$, and $l=l(H)$.
\end{lemma}

\begin{proof}
Let the rank of $B$ be $n-r$. By Sylvester's law of inertia, the quadratic form $B$ can be reformulated as
\begin{equation}\label{}
  B=\sum^{n-r}_{i=1} b_i Y_i^2,
\end{equation}
where $Y_i=\sum^n_{j=1} m_{ij} x_j, i=1,\cdots, n$ are independent linear forms and
\begin{equation*}
  b_1>0, \cdots , b_p>0, b_{p+1}<0, \cdots , b_{n-r}<0.
\end{equation*}
We argue by contradiction.

Suppose that $\sum_{e\in E(H^{\prime})}(|e|-1)>p+l$. We consider the following two systems of linear equations
\begin{equation}\label{}
  Y_1=0, \cdots , Y_p=0, x_i-x_j=0, \ for \ any \ \{i, j\}\in e \in E(H)\setminus E(H^{\prime}),
\end{equation}
and
\begin{equation}\label{}
  x_i-x_j=0, \ for \ any \ \{i, j\}\in e \in E(H).
\end{equation}
Let $c$ be the number of connected components of $H$. By our assumption $\sum_{e\in E(H^{\prime})}(|e|-1)>p+l$, we observe that
\begin{equation*}
  the \ rank \ of \ (14) \leq p+\sum_{e\in E(H)}(|e|-1)-\sum_{e\in E(H^{\prime})}(|e|-1) \leq p+n-c+l-p-l<n-c
\end{equation*}
and, by Lemma \ref{5.8}, the rank of (15) is $\sum_{e\in E(H)}(|e|-1)=n-c$. Hence, there exists a nonzero solution $(x^0_1, \cdots, x^0_n)$ of (14) which fails (15). Then we derive $B(x^0_1, \cdots, x^0_n) \leq 0$ from (14) and $B(x^0_1, \cdots, x^0_n)>0$ from (15), which is a contradiction. This shows $|E(H)| \leq p+l$.

Suppose that $\sum_{e\in E(T)}(|e|-1)<p$. Let us consider the following two systems of linear equations
\begin{equation}\label{}
  x_i-x_j=0, \ for \ any \ \{i, j\}\in e\in E(T), Y_{p+1}=0, \cdots, Y_{n-r}=0
\end{equation}
and
\begin{equation}\label{}
  Y_1=0, \cdots, Y_{n-r}=0.
\end{equation}
Now we compare the ranks of the two systems. By Lemma \ref{5.8} and our assumption that $\sum_{e\in E(T)}(|e|-1)<p$, we estimate by Lemma \ref{5.8}
\[the \ rank \ of \ (16) \leq \sum_{e\in E(T)}(|e|-1)+n-r-p<n-r.\]
On the other hand, the rank of (17) is $n-r$. Therefore, there exists a nonzero solution $(x^0_1, \cdots, x^0_n)$ of (16) which fails (17). Then we derive $B(x^0_1, \cdots, x^0_n)\leq 0$ from (16) and $B(x^0_1, \cdots, x^0_n)>0$ from (17), which is a contradiction. This shows $|E(T)|\geq p$.
\end{proof}

For the number of strong nodal domains $\mathfrak{S}(f)$ of a function $f$, we have the following observation.

\begin{lemma} \label{3.5}
Let $\Gamma=(H, \sigma)$ be a signed hypergraph where $H=(V, E)$ and $f: V\rightarrow \mathbb{R}$ be a function. Let $\Omega=\{v\in V: f(v)\neq 0\}$ be the support set of $f$. Consider the subhypergraph $S=(\Omega, E(S))$ where

$E(S) : =\{e : x_j, x_{j+1} \in e=\{x_i, x_{i+1}, \cdots, x_{i+l-1}\} \in E, f(x_j)sgn(e)f(x_{j+1})>0$ for each $j=i, i+1, \cdots, i+l-2.\}$

Let $T_S=(\Omega, E(T_S))$ be the spanning hyperforest of $S$. Then we have
\begin{equation*}
  \mathfrak{S}(f)=|V|-z-\sum_{e\in E(T_S)}(|e|-1),
\end{equation*}
where $z=|V\setminus \Omega|$ is the number of zeros of $f$.
\end{lemma}

\begin{proof}
By definition, the number $\mathfrak{S}(f)$ is the number of connected components of the subhypergraph $S$ or that of its spanning hyperforest $T_S$. Therefore, we have
\begin{equation*}
  \mathfrak{S}(f)=|\Omega|-\sum_{e\in E(T_S)}(|e|-1)=|V|-z-\sum_{e\in E(T_S)}(|e|-1).
\end{equation*}
\end{proof}

Now, we are ready for the proof of Theorem \ref{5.6}. First we consider the case that $f_k$ has no zeros.

\begin{lemma}\label{5.11}
Let $L, \Gamma=(H, \sigma)$, and $f_k$ be defined as in the Theorem 5.6. If $f_k$ is non-zero at each vertex, then we have
\begin{equation*}
  \mathfrak{S}(f_k)\geq k+r-1-l+l_{+}.
\end{equation*}
\end{lemma}

\begin{proof}
Let us denote the vertex set of $H$ by $V=\{1, 2, \cdots, n\}$. Define $D$ to be the diagonal matrix with $D_{ii}=f_k(i)$ for any $i\in V$, and $B : =D(L-\lambda_kI)D$. By Corollary \ref{4.3}, we get for any function $g: V\rightarrow \mathbb{R}$
\begin{equation}\label{}
  <g, Bg>=\sum_{i, j\in e\in E} deg(x)(-L_{ij})f_k(i)f_k(j)(g(i)-g(j))^2=\sum_{i, j\in e\in E} a_{ij}(g(i)-g(j))^2
\end{equation}
where $a_{ij} := deg(x)(-L_{ij})f_k(i)f_k(j)$, which is nonzero if and only if $\{i, j\}\in e\in E$. Next we apply Lemma \ref{5.10} to the quadratic form $B$ and the hypergraph $H$. Let $H^{\prime}$ be the subhypergraph of $H$ defined as in Lemma \ref{5.10} and $T$ be a spanning hyperforest of $H^{\prime}$. We observe that the edge set of $H^{\prime}$ is exactly the set of edges in $H$ which are $S$-paths of $f_k$ on the signed hypergraph $\Gamma=(H, \sigma)$. Then we derive from Lemma \ref{3.5} that
\begin{equation}\label{}
  \mathfrak{S}(f_k)=n-\sum_{e\in E(T)}(|e|-1),
\end{equation}
Since $D$ is nonsingular, the positive index of inertia of $B$ satisfies
\begin{equation*}
  p_B=p_{D(L-\lambda_k I)D}=p_{(L-\lambda_k I)}=n-(k+r-1).
\end{equation*}
Therefore, we obtain by Lemma \ref{5.10}
\begin{equation}\label{}
  n-(k+r-1)\leq \sum_{e\in E(T)}(|e|-1) \leq \sum_{e\in E(H)}(|e|-1) \leq n-(k+r-1)+l.
\end{equation}
Noticing that
\begin{equation*}
  \sum_{e\in E(H)}(|e|-1)=\sum_{e\in E(T)}(|e|-1)+l(H)=\sum_{e\in E(T)}(|e|-1)+l_{+}(H, \sigma, f_k),
\end{equation*}
we derive
\begin{equation}\label{}
  n-(k+r-1)\leq \sum_{e\in E(T)}(|e|-1) \leq n-(k+r-1)+l-l_{+}.
\end{equation}
Inserting (19) into (21) yields
\begin{equation}\label{}
  k+r-1-l+l_{+} \leq \mathfrak{S}(f_k) \leq k+r-1.
\end{equation}
This proves the lemma.
\end{proof}

Next, we consider the case that every zero of $f_k$  is not in the Fiedler zero set.

\begin{lemma}\label{5.12}
Let $L, \Gamma=(H, \sigma)$, and $f_k$ be defined as in the Theorem \ref{5.6}. If all zeros of $f_k$ lie in $\mathcal{F}^c=\mathcal{F}^c(H, f_k)$. Then we have
\[\mathfrak{S}(f_k)\geq k+r-1-l^{\prime}+l_{+}.\]
where $l^{\prime}=l(H^{\prime})$ and $H^{\prime}$ is the induced subhypergraph of $H$ on the set of nonzeros.
\end{lemma}

{\bf Remark } Due to Proposition \ref{5.3}, we have in the above that $l(H^{\prime})=l(H)+\sum_{y\in Y} d_y$.

\begin{proof}
We denote the vertex set of $H$ by $V=\{1, 2, \cdots, n\}$. Let us denote by $N$ the symmetric matrix obtained from $L$ by deleting all rows and columns with indices from $\mathcal{F}^c$. Then we claim that
\begin{equation*}
  p_{(L-\lambda_k I)}=p_{(N-\lambda_k I)}+|\mathcal{F}^c|,
\end{equation*}
where $p_{(L-\lambda_k I)}$ and $p_{(N-\lambda_k I)}$ are the positive indices of inertia of $L-\lambda_k I$ and $N-\lambda_k I$, respectively. For ease of notation, we do not distinguish $I_n$ and $I_n-|\mathcal{F}^c|$.

We prove this claim by induction with respect to the number $|\mathcal{F}^c|$. When $|\mathcal{F}^c|=0$, the claim holds true. Next, we assume the claim is true when $|\mathcal{F}^c|=m-1$. We consider the case that $|\mathcal{F}^c|=m$. Without loss of generality, we assume $n\in \mathcal{F}^c$ and the matrix $L-\lambda_k I$ has the following form:
\[
L-\lambda_k I=\left(
                  \begin{array}{cc}
                    L_0 & \eta \\
                    \eta^{T} & L_{nn}-\lambda_k \\
                  \end{array}
                \right)
\]
with $\eta^{T}=\{ L_{n1}, \cdots, L_{n, n-1} \}$ and
\[
L_0=\left(
      \begin{array}{cccc}
        L_1-\lambda_k I_{n_1} & 0 & \cdots & 0 \\
        0 & L_2-\lambda_k I_{n_2} & \cdots & 0 \\
        \vdots & \vdots & \ddots & \vdots \\
        0 & 0 & \cdots & L_h-\lambda_k I_{n_h} \\
      \end{array}
    \right)
\]
where $h=d_n$ is the degree of $n$ and $L_i$ is an $n_i\times n_i$ symmetric matrix for each $i$.

Since $n\in \mathcal{F}^c$, there exists an index $j$ such that $L_{nj}f_k(j)\neq 0$. Without loss of generality, we assume $j\in \{1, 2, \cdots, n_1\}$. We set
\begin{equation*}
  u := (f_k(1), \cdots, f_k(n_1), 0, \cdots, 0)^{T}\in \mathbb{R}^{n-1}.
\end{equation*}
Since the vertex $n$ is tree-like, we have $n$ and $s$ are not in same edge, for any $s\in \{1, 2, \cdots, n_1\} \setminus \{j\}$. Therefore, we derive
\begin{equation*}
  \eta^{T}u=\sum^{n_1}_{s=1} L_{ns} f_k{s}= L_{nj} f_k{j} \neq 0.
\end{equation*}
Moreover, we have $L_0 u=0$. Then we can apply Lemma \ref{5.7} to conclude that
\begin{equation*}
  p_{(L-\lambda_k I)}= p_{L_0}+1.
\end{equation*}
By our induction assumption, we have $p_{(L-\lambda_k I)}= p_{L_0}+1=p_{(L-\lambda_k I)}+|\mathcal{F}^c|$. That is, we prove the claim (22).

Let $\mu_1 \leq \cdots\leq \mu_{n-|\mathcal{F}^c|}$ be the eigenvalues of $N$. We assume
\begin{equation*}
  \mu_{k^{\prime}-1} <\lambda_k=\mu_{k^{\prime}}=\cdots=\mu_{k^{\prime}+r^{\prime}-1} < \mu_{k^{\prime}+r^{\prime}}.
\end{equation*}
We observe that $p_{(L-\lambda_k I)}=n-(k+r-1)$, and $p_{(N-\lambda_k I)}=n-|\mathcal{F}^c|-(k^{\prime}+r^{\prime}-1)$. Then (22) implies
\begin{equation}\label{}
  k+r=k^{\prime}+r^{\prime}.
\end{equation}
Note that $H^{\prime}$ is the induced subhypergraph of $N$. By definition, we have $\mathfrak{S}(f_k|H^{\prime})=\mathfrak{S}(f_k)$ since the set of nonzeros stays put. Applying Lemma \ref{5.11} and (23) leads to
\begin{eqnarray*}
  k+r-1&=&k^{\prime}+r^{\prime}-1 \geq \mathfrak{S}(f_k) = \mathfrak{S}(f_k|H^{\prime}) \\
  &\geq& k^{\prime}+r^{\prime}-1-l^{\prime}+l_{+} = k+r-1-l^{\prime}+l_{+}> k+r-1-l^{\prime}+l_{+}.
\end{eqnarray*}
This completes the proof.
\end{proof}

\begin{proof}(Proof of Theorem \ref{5.6}:)
Restrict the function $f_k$ to the induced subhypergraph $\tilde{H}$ of $H$ on $V \setminus \mathcal{F}$. Then $f_k$ is still an eigenfunction of $L|\tilde{H}$ restricting to $\tilde{H}$ corresponding to the eigenvalue $\lambda_k$. We denote by $\mu_1 \leq \cdots\leq \mu_{n-\mathcal{F}}$ the eigenvalues of $L|\tilde{H}$. We assume
\begin{equation*}
  \mu_{\tilde{k}-1} <\lambda_k=\mu_{\tilde{k}}=\cdots=\mu_{\tilde{k}+\tilde{r}-1} < \mu_{\tilde{k}+\tilde{r}}.
\end{equation*}
Observing that all zeros of $f_k|\tilde{H}$ lie in $\mathcal{F}^c(\tilde{H}, f_k|\tilde{H})$, we obtain by Lemma \ref{5.12}
\begin{equation}\label{}
  \mathfrak{S}(f_k)=\mathfrak{S}(f_k|H^{\prime}) \geq \tilde{k}+\tilde{r}-1-l^{\prime}+l_{+}.
\end{equation}
where $l^{\prime}=l(H^{\prime})$. Recall that $H^{\prime}$ is the induced subhypergraph of $H$ on the set of nonzeros.

Applying the interlacing result Lemma \ref{5.9}, we have
\begin{equation*}
  \lambda_{\tilde{k}+\tilde{r}+|\mathcal{F}|} \geq \mu_{\tilde{k}+\tilde{r}} > \mu_{\tilde{k}+\tilde{r}-1}= \lambda_k = \lambda_{k+r-1}.
\end{equation*}
This implies that
\begin{equation}\label{}
  \tilde{k}+\tilde{r}+|\mathcal{F}| \geq k+r.
\end{equation}
Inserting (25) into (24) yields
\[\mathfrak{S}(f_k)\geq k+r-|\mathcal{F}|-1-l^{\prime}+l_{+}.\]
This completes the proof.
\end{proof}

Finally, we give an example to illustrate the lower bound of Theorem \ref{5.6} is sharp.

{\bf Example 2} We consider the signed hypergraph $\Gamma=(H, \sigma)$ given in Figure 1. By numerical computation, we obtain the eigenvalues of $L$ listed below:
\begin{eqnarray*}
  \lambda_1\approx -0.51 \leq \lambda_2\approx 0.22 \leq \lambda_3\approx 0.33 \leq \lambda_4=\lambda_5=\lambda_6=1 \leq \lambda_7\approx 1.95 \leq \lambda_8=\lambda_9=2
\end{eqnarray*}

The smallest eigenvalue $\lambda_1\approx -0.51$ of $L$ is simple and its eigenfunction is
\begin{eqnarray*}
f_1\approx(-0.38, -0.50, -0.38, -0.38, -0.2, -0.38, 0.25, 0.13, 0.25)^{T}.
\end{eqnarray*}
It is direct to figure out $\mathfrak{S}(f_1)=1, l=\sum_{e\in E(H)}(|e|-1)-|V(H)|+c(H)=3, l_{+}=\sum_{e\in E(H_1)}(|e|-1)-|V(H_1)|+c(H_1)=3$. Theorem \ref{5.6} tells $\mathfrak{S}(f_1)\geq 1$, which is sharp.

\section*{Acknowledgements}
This work was supported by the National Natural Science Foundation of China (No. 11971164) and the Qinghai Natural Science Foundation of China (No. 2022-ZJ-973Q).

\vskip 3mm
\end{spacing}
\end{document}